\documentclass[12pt,reqno]{amsart}
\usepackage[utf8]{inputenc}
\usepackage{amsmath}
\usepackage{amsfonts}
\usepackage{amssymb}
\usepackage{amsthm}
\usepackage[left=3cm,right=3cm,top=3cm,bottom=3cm]{geometry}
\usepackage{xcolor}
\usepackage{comment}

\usepackage{lmodern,amssymb,bbm}
\usepackage{hyperref}
\hypersetup{linkcolor={blue},citecolor={blue}}

\theoremstyle{plain}    
 \newtheorem{theorem}{Theorem}[section]
 \numberwithin{equation}{section} 
 \numberwithin{figure}{section} 
 \theoremstyle{plain}
 \theoremstyle{plain}    
 \newtheorem{corollary}[theorem]{Corollary} 
 \theoremstyle{plain}    
 \newtheorem{proposition}[theorem]{Proposition} 
 \theoremstyle{plain}    
 \newtheorem{lemma}[theorem]{Lemma} 
 \theoremstyle{remark}
 \newtheorem{remark}[theorem]{Remark}

 \theoremstyle{definition}

\theoremstyle{definition}
\newtheorem{definition}[theorem]{Definition}

\usepackage{enumitem}

\newcommand{\de}{\delta}
\newcommand{\e}{\epsilon}
\newcommand{\om}{\omega}

\newcommand{\f}{\varphi}
\newcommand{\p}{\partial}

\newcommand{\dd}{i\partial \bar{\partial}}
\newcommand{\tr}{{\rm tr}}

\newcommand{\J}{{\mathcal{J}}}
\newcommand{\M}{{\mathcal{M}}}

\newcommand{\La}{\Lambda}

\address{Institut de Math\'ematiques de Jussieu-Paris Rive Gauche \\ Sorbonne Universit\'e \\ 4 place Jussieu \\  75005 Paris, France}

\email{\href{mailto:tat-dat.to@imj-prg.fr}{tat-dat.to@imj-prg.fr}}
\urladdr{\href{https://sites.google.com/site/totatdatmath/home}{https://sites.google.com/site/totatdatmath}}

\title[Degenerate  J-flow]{Degenerate  J-flow on compact K\"ahler manifolds}
\author{Tat  Dat T\^o}
\date{\today}
\begin{document}
\begin{abstract}
In this  note, we study a degenerate twisted J-flow on compact K\"ahler manifolds. We show that it  exists for all time, it is unique and converges to a weak solution of a degenerate twisted J-equation.   In particular, this confirms an expectation formulated by  Song-Weinkove for the J-flow. As a consequence,  we  establish the properness of the Mabuchi K-energy twisted by a certain semi-positive closed (1,1)-form  for K\"ahler classes in a certain subcone.
\end{abstract}

\maketitle

\section{introduction}
Let $(X, \omega)$ be a compact K\"ahler manifold of complex dimension $n$.  We define the space $\mathcal{H}_\omega$  of K\"ahler potentials  by 
$$\mathcal{H}_\omega = \{ \f\in C^\infty(X): \omega_\f:= \omega+\dd \f>0\}.$$
 Let $\theta$ be another K\"ahler metric. The  J-flow is the parabolic flow defined on  $\mathcal{H}_\om$ by
\begin{equation}
     \dfrac{\p \f }{\p t}= c-\dfrac{n\omega_{\f }^{n-1}\wedge \theta }{\omega_\f ^n} = c-\tr_{\omega_\f}\theta, \quad \f|_{t=0}=\f_0,
\end{equation}
 where $\f_0\in \mathcal{H}_\omega$ and  $c$ is the constant given by
 $$c=\dfrac{n[\theta]\cdot [\omega]^{n-1}}{[\omega]^n}.$$ 
A stationary point $u$ of the J-flow defines a K\"ahler metric   $\tilde \omega=\omega+i\p\bar \p u  \in [\omega]$  satisfying  the J-equation
$$\tilde\omega^{n-1}\wedge \theta = c\tilde\omega^{n} .$$

The J-flow was introduced by Donaldson \cite{Don99} in the setting of moment maps and by Chen \cite{Che00} as the gradient flow of the $\J$-functional appearing in the formula of the Mabuchi K-energy. The critical points of the Mabuchi K-energy are constant scalar curvature Kähler (cscK) metrics.  The existence of cscK metrics in a given K\"ahler class is a fundamental problem in K\"ahler geometry and has been extensively studied. 
 Tian \cite{Tian97} conjectured that the properness of the Mabuchi  K-energy implies the existence cscK  metrics. This conjecture was proved recently by Chen-Cheng  \cite{CC21a, CC21b}.

In \cite{SW08},  Song-Weinkove pointed  that $\J$ being bounded from below is sufficient to imply the properness of the Mabuchi  K-energy. A uniform bound from below for the $\J$-functional can be obtained as long as the corresponding J-flow exists for all time and converges.  For this reason  the J-flow   has   received  a lot  of attention in order to determine  conditions  for its convergence.   Necessary and sufficient conditions  for the regular case (i.e. when $\omega$ and $\theta$ are K\"ahler) have been obtained by Song-Weinkove \cite{SW08} extending previous results by Weinkove \cite{Wein04,Wein06}. An algebraic understanding of these conditions has been developed in  \cite{ LS15, CS,  GaoChen, Son20, SD}.

In the degenerate case, 
  conditions for the existence and convergence of {\sl weak J-flows} in dimension two have obtained in \cite{FLSW, SW12}. In particular, Song-Weinkove \cite{SW12} assumed that $\theta$ is only semi-positive, positive away from a divisor $D $ and big,  then under certain  condition the  weak J-flow  exists,  is smooth away from $D$ and converges in $C^\infty$ away from $D $ to a weak solution of a degenerate J-equation.  They  expected that this result can be extended to higher dimension (cf. \cite[Remark 3.1]{SW12}).

 \medskip
 In \cite{Zhe15},  Zheng  suggested to study a twisted version of the J-flow defined by
 \begin{equation}
     \dfrac{\p \f }{\p t}= c_\beta-\dfrac{n\omega_{\f }^{n-1}\wedge \theta }{\omega_\f ^n}  -\beta \frac{\omega^n}{\omega_\f^n}= c_\beta-\tr_{\omega_\f}\theta -\beta \frac{\omega^n}{\omega_\f^n}, \quad \f|_{t=0} =\f_0,
\end{equation}
 where $\beta\in [0,+\infty)$,  $\f_0\in \mathcal{H}_\omega$ and  $c_\beta$ is the constant given by
 $$c_\beta=\dfrac{n[\theta]\cdot [\omega]^{n-1}}{[\omega]^n}+ \beta.$$ 
 When $\beta=0$, we obtain the J-flow.
In \cite{Zhe15},  Zheng proved the existence and convergence under a positivity condition (see \eqref{eq:cond_2} below).  
  A main advantage of this twisted flow is that the  necessary and sufficient condition for the existence and convergence for the twisted J-flow is weaker than the one for the  J-flow, while it may still imply the properness of the Mabuchi K-energy. We will elaborate on this point in section \ref{sect:Mabuchi}.

\medskip
In this note, motivated by  previous works of Song-Weinkove \cite{SW08, SW12} and Zheng \cite{Zhe15},  we study the existence and convergence for the degenerate twisted J-flow where $\theta$ is no longer a K\"ahler metric,  but a closed $(1,1)$-form satisfying a certain nonnegativity condition. We  extend the results by Song-Weinkove \cite{SW12} for the J-flow on K\"ahler surfaces mentioned above  to the twisted J-flow in any dimension.

In particular, this confirms the  expectation formulated in \cite{SW12} for the untwisted case: there exists a unique degenerate J-flow  for all times which is smooth away from a Zariski set and  converges to a weak solution of the degenerate  J-equation.

More precisely, we assume that the form $\theta$ is merely semi-positive and satisfies the following condition.  We fix an effective divisor $D$ on $X$ and $h$ a Hermitian metric on the line bundle $\mathcal{O}_X(D)$. Let $s$ be a holomorphic section of $\mathcal{O}_X(D)$  and assume that 
\begin{equation}\label{eq:cond_1}
\theta\geq C_0 | s|^{2\gamma} \omega, \text{ and } \theta-\varepsilon_0 R_h\geq C_0 \omega,
\end{equation}
for some positive constants $\gamma, C_0, \varepsilon_0$, where $R_h=-i\p\bar \p \log h$ denotes the curvature form of the Hermitian metric $h$. These conditions are for example satisfied for certain 
 $\theta\in c_1(K_X)$ with $K_X$ is big and nef,  i.e $X$ is a smooth minimal model of general type. 

\medskip
We shall moreover assume  that   there is $\hat \omega\in [\omega]$ such that 
\begin{equation}\label{eq:cond_2}
    (c_\beta\hat \omega-(n-1)\theta)\wedge \hat\omega^{n-2}>0,
\end{equation}
as an $(n-1, n-1)$ form,
where  $$c_\beta =\dfrac{n[\theta]\cdot [\omega]^{n-1}}{[\omega]^n}+\beta.$$
In \cite{Zhe15}, the author proved that if $\theta>0$, the condition \eqref{eq:cond_2} is equivalent to the long time existence and the smooth convergence of the twisted J-flow. This result extends the previous one of Song-Weinkove \cite{SW08} for the J-flow. In particular, the condition \eqref{eq:cond_2} is also equivalent to the existence of a $C$-subsolution introduced by Sz\'ekelyhidi \cite{Sze} and Guan \cite{Guan} for the twisted J-equation as well as a parabolic $C$-subsolution introduced by Phong-T\^o \cite{PT} for the twisted J-flow. As explained in \cite{SW08,Zhe15}, when $K_X$ is ample and $\theta$ is a K\"ahler metric in $-c_1(K_X)$ this condition then implies the properness of the  Mabuchi K-energy on $\mathcal{H}_\om$.  We refer to \cite{LSY,Der,JSS,SD2} for recent works  studying sufficient conditions so that the Mabuchi  K-energy  is proper.

\medskip
Let $PSH(X,\omega)$ denote the set of $\omega$-plurisubharmonic functions: these are functions $\phi: X\rightarrow \mathbb{R}\cup \{-\infty\}$ which are locally given as the sum of a smooth and a plurisubharmonic function,  such that $\om+dd^c\phi\geq 0$ in the weak sense of currents and $\phi \not\equiv  -\infty$. We  define \begin{equation}
    \mathcal{H}^{weak}_\omega = \{ \f\in PSH(X,\omega)\cap L^{\infty}(X)\cap C^{\infty}(X\setminus D): \omega+\dd \f>0 \text{ on } X\setminus D \}.
\end{equation}
Our main result is the following.

\begin{theorem}\label{thm:1}
Let $(X, \om)$ be a compact K\"ahler manifold and $\beta$ be  a non-negative constant. Assume that $\theta$ is a closed semi-positive  $(1,1)$-form satisfying  conditions \eqref{eq:cond_1}
 and \eqref{eq:cond_2}.
 
 For any $\f_0\in \mathcal{H}_{\omega}$ and for  all $t>0$ there exists a unique solution $\f=\f(t)  \in \mathcal{H}^{weak}_\omega$ to the degenerate twisted J-flow
\begin{equation}
     \dfrac{\p \f }{\p t}= c_\beta-\dfrac{n\omega_{\f }^{n-1}\wedge \theta }{\omega_\f ^n}  -\beta \frac{\omega^n}{\omega_\f^n}, \quad \f|_{t=0}=\f_0.
\end{equation}
Moreover, $\f(t)\rightarrow \f_\infty$ in $C^\infty_{loc}(X\setminus D)$ as $t\rightarrow \infty$, where $\f_\infty \in 
 \mathcal{H}^{weak}_\omega$   satisfies
$$ c_\beta \omega_{\f_\infty}^{n}  =  n\omega_{\f_\infty}^{n-1}\wedge\theta+\beta \om^n.$$
\end{theorem}
When $\beta=0$, this theorem confirms the expectation of Song-Weinkove \cite[Remark 3.1]{SW12}.

 \medskip
 Theorem \ref{thm:1}
also provides  a uniform lower bound for the twisted $\J$-functional $\J^\theta_{\omega,\beta}$. 

 \begin{corollary}\label{cor:j_bounded_below}
Let $(X, \om)$ be a compact K\"ahler manifold and $\beta$  a non-negative constant. Assume that $\theta$ is a closed semi-positive  $(1,1)$-form satisfying  conditions \eqref{eq:cond_1}
 and \eqref{eq:cond_2}.
  Then there exists a constant $B$ depending only on $X,\omega,\theta$ such that 
  $$\J^\theta_{\omega, \beta}(\f)\geq B, \quad \forall\f \in \mathcal{H}_\omega.$$
\end{corollary}

Finally, we have the following criterion for the properness of the  Mabuchi K-energy twisted by a semipositive  $(1,1)$-form. 
\begin{corollary}\label{cor:twisted_cscK}
Let $X$
 be a compact K\"ahler manifold and $\eta$ 
 be a smooth closed semipositive  $(1,1)$-form. Assume that  $\theta\in -c_1(X)+ [\eta]$  is a semi-positive closed $(1,1)$-form satisfying \eqref{eq:cond_1}.
Suppose that 
\begin{equation}\label{eq:proper_ent}
    \M_{ent}(\f)\geq \alpha_0 J_\om (\f)-C,
\end{equation}
for some $\alpha_0>0$,
where $\M_{ent}$ is the entropy part in the formula of the Mabuchi functional and $J_\om$ is the Aubin's J-functional (see section \ref{sect:Mabuchi}).   
For any $\beta \in [0, \alpha_0 )$, denote by $\mathcal{C}$  the cone of all K\"ahler classes $[\omega_0]$ such that there exists a K\"ahler form $\omega \in [\omega_0]$ satisfying

\begin{equation}\label{eq:cond_2_general_type}
    \left( \left[ \frac{n(-c_1(X)+[\eta])\cdot [\omega]^{n-1}}{[\omega]^n}+ \beta \right]\omega -(n-1)\theta \right)\wedge  \omega^{n-2}>0,
\end{equation}
as an $(n-1, n-1)$-form.

Then the  Mabuchi functional twisted by $\eta$ is proper on every class in  $\mathcal{C}$.  Therefore by \cite[Theorem 4.1]{CC21b} there exists a twisted  cscK metric in $[\om]$ which satisfies 
\begin{equation}
    S(\om_\f)= \underline{S}+ \tr_{\om_\f}\eta -\underline{\eta},
\end{equation}
where
$$\underline{S} = n\frac{c_1(X).[\omega]^{n-1}}{[\omega]^n} \quad \text{and }\quad \underline{\eta} = n\frac{[\eta].[\omega]^{n-1}}{[\omega]^n}.$$
\end{corollary}

It follows from \cite[p.95]{Tian00} that \eqref{eq:proper_ent} holds for  $\alpha_0=\alpha_X([\om])$, namely  Tian's alpha invariant.  Corollary \ref{cor:twisted_cscK} thus  gives a criterion in line with recent results due to \cite{Der, LSY, Zhe15,SD} (see Proposition \ref{prop:twisted_cscK_crt_2} and Remark \ref{remark_final}). In particular, when $\eta=0$, this extends the result in \cite{Zhe15} for the properness of the Mabuchi K-energy on compact K\"ahler manifolds with ample canonical bundle  to  minimal models of general type.

A  key ingredient of the proof of Theorem \ref{thm:1} is a uniform $C^0$ estimate for the weak solution of the degenerate twisted J-flow.  In \cite{SW12}, for K\"ahler surfaces,  the authors reduced the degenerate J-equation to a degenerate Monge-Amp\`ere equation then obtained a uniform $C^0$ bound using  Yau's estimate \cite{Yau}. It seems delicate  to extend this approach to higher dimension.

In this work,  we adapt a $C^0$ estimate from \cite{PT} which is based on the parabolic Alexandrov-Bakelmann-Pucci (ABP) maximum principle (see  \cite{Sze} for the elliptic version). We can not apply the $C^0$ estimate from \cite{PT} directly since $\theta$ is assumed to be {\sl positive} in  \cite{PT} while it is  merely semi-positive in our case.  In \cite{PT}, with $\theta>0$, the parabolic equation has been reformulated as $$\partial_t \f =f(\lambda_\theta(\om_\f)),$$
where $\lambda_\theta( \om_\f) \in \mathbb{R}^n$ is the vector of eigenvalues of $\om_\f$  in  normal coordinates with respect to $\theta$.
In our case,   we use  normal  coordinates for $\omega$ instead of $\theta$,  so that $\omega_\f $ is diagonal with entries $\lambda_1,\ldots, \lambda_n$.  The  twisted J-flow can then be written in a new form as
$$  \p_t \f =c_\beta- \sum _{k=1}^n\dfrac{\mu_k}{\lambda_k} -\frac{\beta}{\lambda_1\ldots \lambda_n},$$
where    $ \mu_1, \ldots ,\mu_n$ denote the diagonal elements of $\theta$ in these coordinates. Then the $C^0$ estimate is deduced using the ABP estimate, a Harnack type inequality as well as a local $C^2$ estimate under the positivity condition \eqref{eq:cond_2} (cf. Lemma \ref{lem:key_C0}).

Another ingredient of the proof of Theorem 1.1 is a uniform $C^2$ estimate away from  a divisor (Lemma \ref{lem:C2}). We use the maximal principle with a well-known trick due to Tsuji \cite{Tsu},  but we do not  use the Phong-Sturm's trick (cf. \cite{PS10}) as in \cite{SW12}. In our case, some extra terms also appear in the computation due to the new term in the twisted J-flow. 

\medskip
The same idea for $C^0$-estimate  can also be used to  give a uniform estimate  for the  twisted $\sigma_{n-1}$ inverse equation   where the the uniform bound depends on  $\|\theta\|_{C^0(X,\om)}$ but not on  $\|\om\|_{C^0(X,\theta)}$ as in \cite{Sze}.

\begin{proposition}
Let $(X,\om)$ be a compact K\"ahler manifold  and $\theta$ be another K\"ahler metric. Let  $\psi, \rho \in C^\infty (X)$ be smooth   functions on $X$ with $\psi>0$. 
Assume that there is $\omega'=\om+\dd \underline{u}\in [\omega]$ such that 
\begin{equation}
    (\psi\omega'-(n-1)\theta)\wedge \omega'^{n-2}>0,
\end{equation}
as an $(n-1, n-1)$ form. 
Assume $u$ is a smooth solution
of the equation \begin{equation}\label{eq:j_eq}
    n\omega^{n-1}_u\wedge\theta + \rho \om^n  =\psi \om _u^{n}, \quad \sup_X (u-\underline{u})=0.
    \end{equation}
Then there exists a constant $C>0$ which only depends  on $X,\omega,\|\underline{u}\|_{C^2(X,\om)},  \|\psi\|_{L^\infty(X)},$ $ \|\rho\|_{L^\infty(X)}$  and $\|\theta\|_{C^0(X,\om)}$ such that $$\|u\|_{L^\infty(X)}\leq C.$$ 
\end{proposition}

\medskip
\noindent
{\bf Acknowledgements.} 
The author is grateful to Vincent Guedj  for support, suggestions and encouragement. We also would like to thank Duong Hong Phong and   Zakarias Sjöström Dyrefelt for very useful discussions.
The author is  partially supported by   ANR-21-CE40-0011-01 (research project MARGE) and  ANR-11-LABX-0040 (research project HERMETIC). The author would like to thank the referee for useful comments and suggestions.

\section{Preliminaries}
\subsection{Mabuchi K-energy} \label{sect:Mabuchi}
Let $(X,\om)$ be a $n$-dimensional compact K\"ahler manifold.  We define the space $\mathcal{H}_\omega$ of K\"ahler potentials  by
$$\mathcal{H}_\omega = \{ \f\in C^\infty(X): \omega_\f:= \omega+\dd \f>0\}.$$
 The Mabuchi K-energy functional  was introduced by Mabuchi  \cite{Mab87} which is characterized by 
\begin{equation}
    \dfrac{d \M_\om(\f_t)}{ dt} =  - \int_X\dot \f_t  (S_{\f_t} - \underline{S} ) \dfrac{\omega^n_{\f_t}}{n!},
\end{equation}
for any path $(\f_t)$ in $\mathcal{H}_\omega$,
where $S_{\f_t}$ is the scalar curvature of $\om_{\f_t}$ and
$\underline{S} = n\frac{c_1(X).[\omega]^{n-1}}{[\omega]^n}.$
 The critical points of the Mabuchi K-energy are {\sl constant scalar curvature Kähler (cscK) metrics}.
 
The Chen-Tian  formula (see \cite{Che00, Tian00}),  expresses the Mabuchi K-energy  as the sum of an entropy part and a pluripotential part 

$$\mathcal{M}_\om(\f)=\M_{ent}(\f)+\M_{pp}(\f),$$
where 
$$\M_{ent}(\f)= \int_X \log \left(\frac{\omega_\f^n}{\omega^n} \right) \frac{\omega^n_\f  }{ n!} $$
 and 
 $$\M_{pp}(\f)=\J^{-Ric(\omega)}_\omega .$$
Here, the twisted $\J$-functional $\J^{\theta}_{\omega}$ associated to a $(1,1)$-form $\theta$ is defined by   

\begin{eqnarray} \label{eq:J_fct}
     \J^\theta_\omega(\f)&=&   \int_0^1\int_X \dot \f_t(n \theta\wedge \omega_{\f_t}^{n-1}  -\underline{\theta} \omega_{\f_t}^n)  \frac{dt}{n!}\\
     &=& \frac{1}{n!}\int_X \f \sum_{k=0}^{n-1}\theta \wedge \omega^k\wedge\omega_{\f}^{n-1-k}  -\frac{1}{(n+1)!}\int_X \underline{\theta} \f \sum_{k=0}^n \omega^k\wedge \omega^{n-k}_\f, \nonumber
\end{eqnarray}
where  $(\f_t)$ is any path in  $\mathcal{H}_\omega$ between 0 and $\f$ and
 $$\underline{\theta}:=\dfrac{n[\theta]\cdot [\omega]^{n-1}}{[\omega]^n}.$$

Let $\eta\geq 0$ be a smooth closes $(1,1)$-form, the {\sl twisted Mabuchi K-energy} is defined by
$\M_{\om}^\eta:= \M_{\om}+ \J^{\eta}_\om$. The critical points of $\M^\eta_\om$ satisfy the following equation
\begin{equation}\label{eq:tcscK}
    S(\om_\f)= \underline{S}+ \tr_{\om_\f}\eta -\underline{\eta},
\end{equation}
which defines  {\sl twisted cscK metrics} in $[\om]$ (cf. \cite{Fine04,Fine06,CC21b}).

A functional $\mathcal{F}$ is called  {\it coercive} if there are constants $C_1, C_2$ such that 
 \begin{equation}\label{cond:proper_def}
     \mathcal{F}(\f)\geq C_1 J_\omega (\f)-C_2,
 \end{equation}
  for all $\f \in \mathcal{H}_\omega$,  where $J_\om$ is     Aubin's  functional defined by
   \begin{eqnarray}
    J_\omega(\f) 
    &= & \int_0^1\int_X \dot \f_t(\om^n-\om^n_{\f_t}) \frac{dt}{n!}=\int_X \f \frac{\omega^n}{n!} - \frac{1}{(n+1)!}\int_X  \f \sum_{k=0}^n \omega^k\wedge \omega^{n-k}_\f,
\end{eqnarray}
where $(\f_t)$ is any path in  $\mathcal{H}_\omega$ between 0 and $\f$.

Coercivity is a strong form of  ``properness''.
The cscK metrics are critical points of the Mabuchi K-energy. 
  Tian \cite{Tian97} conjectured that the properness of Mabuchi functional implies the existence of cscK metrics. This conjecture was proved recently by Chen-Cheng  \cite{CC21a,CC21b}. The similar result for the existence of twisted cscK metrics was also proved  in \cite[Theorem 4.1]{CC21b}.

\begin{definition}
A functions $\phi: X\rightarrow \mathbb{R}\cup \{-\infty\}$ is {\it $\omega$-plurisubharmonic} ($\om$-psh for short) if it is locally given as the sum of a smooth and a plurisubharmonic function, and such that $\om+dd^c\phi\geq 0$ in the weak sense of currents.
Let $PSH(X,\omega)$ denote the set of all $\omega$-plurisubharmonic functions which are not identically $-\infty$. 
\end{definition} 
  
\noindent We remark that all functionals  above are also well-defined on $PSH(X, \om)\cap L^\infty(X)$. 
  
  \medskip
We now recall an observation due to Zheng \cite{Zhe15}.   Suppose that $\alpha_0$ is a   positive number satisfying
\begin{equation}\label{eq:proper_ent_1}
    \M_{ent}(\f)\geq \alpha_0 J_\om (\f)-C, \forall \f\in \mathcal{H}_\om
\end{equation}
  for some $C>0$, 
  then for any $\beta\in [0,\alpha_0)$ we have
  \begin{eqnarray}
       \M(\f)&\geq& \alpha_0 J_\om (\f)-C+ \J^{-Ric(\om)}_\om
       \\
       &=& (\alpha_0 -\beta) J_\om (\f) -C +\J^{-Ric(\om)}_{\om,\beta}\\
       &\geq& (\alpha_0 -\beta) J_\om (\f)  -C+ \inf_{\f\in \mathcal{H}_\om} \J^{-Ric(\om)}_{\om,\beta} ,
  \end{eqnarray}
  where $\J^{-Ric(\om)}_{\om,\beta}=\J^{-Ric(\om)}_\om+\beta J_\om.$
Therefore, if $\J^{-Ric(\om)}_{\om,\beta}$ is  uniformly bounded from below on $\mathcal{H}_\om$ then the Mabuchi K-energy is coercive. 

We remark that  \eqref{eq:proper_ent_1} holds for  $\alpha_0=\alpha_X([\om])$, namely Tian's alpha invariant (cf. \cite[p.95]{Tian00} or \cite[Lemma 4.1]{SW08}), defined by
\begin{equation*}
    \alpha_X([\om])=\sup\{\alpha>0: \exists C>0 \text{ such that } \int_Xe^{-\alpha(\f-\sup_X\f)}\om^n\leq C, \forall \f\in PSH(X,\om)\}.
\end{equation*}

In general, we define the functional  $  \J^{\theta}_{\om,\beta}:= \J^\theta_{\om} + \beta J_\om $. 
The critical point of $\J^{\theta}_{\om,\beta}$ satisfies a new fully nonlinear equation in $\mathcal{H}_\om$, namely  the {\it twisted J-equation}:

$$n\omega_\f^{n-1}\wedge \theta+ \beta\om^n = c_\beta\omega_\f^{n},$$ 
where
$$c_\beta=\dfrac{n[\theta]\cdot [\omega]^{n-1}}{[\omega]^n}+ \beta.$$ 

We also remark that this equation is similar to the equation (4) in \cite{CS} in which $\beta \om^n $ replaced by $\beta\theta^n$ assuming $\theta>0$. Moreover, this is also a special case of a general twisted J-equation introduced recently by Chen \cite{GaoChen}, Theorem 1.11, with $f=\beta\frac{\om^n}{\theta^n}$ and $\theta>0$. 

\subsection{Twisted J-flow}\label{section:twisted_j_flow}
Let $(X,\omega)$ be a K\"ahler manifold of dimension $n$ and  $\theta$ be a  smooth $(1,1)$-forms on $X$. For $\beta\in [0,\infty)$, consider the twisted  J-flow
\begin{equation}\label{eq:twisted_j_flow_0}
     \dfrac{\p \f }{\p t}= c_\beta-\dfrac{n\omega_{\f }^{n-1}\wedge \theta }{\omega_\f ^n}  -\beta \frac{\omega^n}{\omega_\f^n}, \quad \f|_{t=0} =\f_0.
\end{equation}
where $$c_\beta=\dfrac{n[\theta]\cdot [\omega]^{n-1}}{[\omega]^n}+\beta.$$ 
This  is the gradient flow for the twisted $\J$-functional  $\J^\theta_{\omega,\beta}=\J^\theta_{\om} + \beta J_\om$ above. 

\medskip
When $\theta>0$, we have the following result is due to Song-Weinkove \cite{SW08} for $\beta=0$ and Zheng \cite{Zhe15} for $\beta\geq 0$:
\begin{theorem}[\cite{SW08,Zhe15}]\label{thm:SW-Z}
Let $(X,\omega)$ be a K\"ahler manifold of dimension $n$ and  $\theta$ be a  closed {\sl positive } $(1,1)$-forms on $X$.
 Assume that there exists a metric $\hat \om\in [\omega]$ such that 
 \begin{equation}
    (c_\beta\hat \omega-(n-1)\theta)\wedge \hat\omega^{n-2}>0,
\end{equation}
where  $$c_\beta =\dfrac{n[\theta]\cdot [\omega]^{n-1}}{[\omega]^n}+\beta.$$
Then for any $\f_0\in \mathcal{H}_{\omega}$, the twisted J-flow \eqref{eq:twisted_j_flow_0} admits a unique solution $\f\in C^\infty(X\times[0, \infty))$, with $\f(t)  \in \mathcal{H}_\omega, \forall t\geq 0$.  

Moreover, $\f(t)\rightarrow \f_\infty$ in $C^\infty(X)$ as $t\rightarrow \infty$, where $\f_\infty \in 
 \mathcal{H}_\omega$   satisfies
$$ c_\beta \omega_{\f_\infty}^{n}  =  n\omega_{\f_\infty}^{n-1}\wedge\theta+\beta \om^n.$$
\end{theorem}

\medskip
 We are now interested in the degenerate case where $\theta$ is merely semi-positive. 
We  have the following  local expression  for the degenerate twisted J-flow which will be used to construct  weak solutions.
We use  normal  coordinates for $\omega$ so that $\omega_\f $ is diagonal with entries $\lambda_1,\ldots, \lambda_n$. Let  $ \mu_1, \ldots ,\mu_n$ denote the diagonal elements of $\theta$ in these coordinates. The twisted J-flow can then be written as
$$  \p_t \f =c_\beta- \sum _{k=1}^n\dfrac{\mu_k}{\lambda_k} -\frac{\beta}{\lambda_1\ldots \lambda_n},$$
or $$\sum _{k=1}^n\dfrac{\mu_k}{\lambda_k}+\frac{\beta}{\lambda_1\ldots \lambda_n}=c_\beta-\p_t \f.$$
If there is a uniform constant $\delta_0>0 $ such that

\begin{equation}\label{eq:pos_cond_e}
   c_\beta \omega^{n-1}-(n-1)\omega^{n-1}\wedge \theta \geq \delta_0 \omega^{n-1},
\end{equation}
then in the  coordinates above, this inequality can be rewritten  as
\begin{equation}\label{eq:pos_condition_2}
c_\beta -\sum_{j\neq k} \mu_j \geq \delta_0. 
\end{equation}

\section{Uniform estimates}\label{sect:est}
\subsection{Approximating the degenerate twisted J-flow} Let $(X,\om)$ be a compact K\"ahler manifold and $\theta$ be a smooth $(1,1)$-form. 
We first assume that  
\begin{equation}\label{eq:theta_semipositive}
 \theta\geq 0 \text{ and  } \int_X \theta\wedge \om^{n-1}>0.   
\end{equation}
Set $\theta_\e= \theta+\epsilon\omega$, for $\e\geq 0$  and,  define 
 $$c_{\e}:=\dfrac{n[\theta_\e]\cdot [\omega]^{n-1}}{[\omega]^n}+\beta.$$
 Then $\theta_\e$ is positive for all $\e> 0$. 
 By our assumption in Theorem \ref{thm:1},  there exists $\hat \omega\in [\omega]$ such that
 \begin{equation}\
    (c_\beta\hat\omega-(n-1)\theta)\wedge \hat\omega^{n-2}>0. 
\end{equation}
As  $\e\rightarrow 0$, $c_{\e}\rightarrow c_\beta$, so we may choose $\e_0>0$ such that for $\e\in [0,\e_0]$ we have
\begin{equation}\label{eq:cond_2_e}
    (c_{\e}\hat\omega-(n-1)\theta_\e)\wedge \hat\omega^{n-2}\geq \delta_0\hat\omega^{n-1}>0. 
\end{equation}
for a uniform constant $\delta_0>0$ independent of $\e$. 
Since $\hat\omega\in[\omega]$, $\hat\omega=\omega+\dd \phi$ for some $\phi\in C^\infty(X)$, we can replace $\omega$ by $\hat\omega$ and still denote it by $\omega$ to get the inequality:
\begin{equation}\label{eq:cond_2_red_e}
    (c_{\e}\omega-(n-1)\theta_\e)\wedge \omega^{n-2}\geq \delta_0\omega^{n-1}. 
\end{equation}

 We now consider  the family of twisted J-flows
\begin{equation}\label{eq:J-flow_e}
    \dfrac{\p \f_\e }{\p t}= c_{\e}-\dfrac{\omega_{\f_\e }^{n-1}\wedge \theta_\epsilon }{\omega_{\f_\e} ^n}-\beta \frac{\omega^n}{\omega_{\f_\e}^n}, \quad \f|_{t=0}=\f_0,
\end{equation}
for $\beta\in [0,\infty)$ and $\e\in(0,\e_0]$.

\medskip
It follows from Theorem \ref{thm:SW-Z}  that  for any $\e\in (0,\e_0]$, with the condition \eqref{eq:cond_2_red_e},  the twisted J-flow \eqref{eq:J-flow_e} exists for all time and converges to the unique solution of the twisted J-equation 
\begin{equation}
    \omega_{u_\e}^{n-1}\wedge \theta_\e+\beta\om^n = c_{\e}\omega_{u_\e}^{n}  . 
\end{equation}
In the next sections we prove uniform $L^\infty$ bounds for $\f_\e, \dot\f_\e$  on $X$  when $\theta$ is  semi-positive with $ \int_X \theta\wedge \om^{n-1}>0$,  and $C^\infty$ estimates for $\f_\e$ away from $D$ when we assume further that $\theta$ satisfies the condition \eqref{eq:cond_2}.
\subsection{Estimate for the time  derivative}
The $C^1$ estimate for  $\f_{\e}$ in the time variable follows easily from the maximal principle.
\begin{lemma}\label{lem:c1_t}
There exists a uniform constant $C$ such that 
$$\| \p_t \f_\e\|_{L^{\infty}(X)}\leq C$$
and 
\begin{equation}\label{eq:chi_omega_e}
\omega_{\f_\e}\geq \frac{1}{C}\theta_\e. 
\end{equation}
\end{lemma}
\begin{proof}
 Differentiating \eqref{eq:J-flow_e}, we get
\begin{equation}
    \frac{\p }{\p t}\left(  \frac{\p \f_\epsilon}{\p t}\right)= \mathcal{L} \left(  \frac{\p \f_\epsilon}{\p t}\right),
\end{equation}
where $$\mathcal{L}= \left( \om_\f^{k\bar j} \om_\f^{i\bar \ell} (\theta_\e)_{k\bar \ell}+\beta\frac{\om^n}{\om_\f^n}\om_\f^{i\bar j}\right) \p_i\p_{\bar j}.$$
Here we write $\omega= \sum_{k,\ell}i\om_{k\bar \ell}dz^k\wedge d\bar z^{\ell}$ in local coordinates and $(\om^{k\bar\ell})= (\om_{k\bar\ell})^{-1}$.
We infer from the maximum principle that
\begin{equation}
    \min_X \left(c_\e-\dfrac{\omega_{\f_0 }^{n-1}\wedge \theta_\epsilon }{\omega_{\f_0} ^n}-\beta\frac{\om^n}{\om_{\f_0}^n}\right) \leq \p _t \f_\epsilon \leq \max_X \left(c_\e -\dfrac{\omega_{\f_0 }^{n-1}\wedge \theta_\epsilon }{\omega_{\f_0} ^n}-\beta\frac{\om^n}{\om_{\f_0}^n}\right).
\end{equation}
This implies that $|\p_t \f_\epsilon|\leq C,$ hence $\tr_{\omega_{\f_\e}}\theta_\e\leq C$  for a uniform constant $C>0$, therefore  \begin{equation}
\omega_{\f_\e}\geq \frac{1}{C}\theta_\e. 
\end{equation}
\end{proof}

\subsection{$C^0$ estimate}
We provide here a $C^0$ estimate using the ABP maximum principle adapting arguments from \cite{PT}. Our situation is different from the general setting in \cite{PT} where $\theta$ is assumed to be {\sl positive}.  In the latter case, the parabolic equation has been reformulated as $$\dot \f =f(\lambda_\theta(\om_\f)),$$
where $\lambda_\theta( \om_\f) \in \mathbb{R}^n$ is the vector of eigenvalues of $\om_\f$   with respect to $\theta$. 

\medskip
Since we work with  a  smooth $(1,1)$-form $\theta$ which is only semipositive,  this setting is no longer valid.  We consider instead $\lambda=\lambda _\om ( \om_{\f_\e})$ and $\mu=(\mu_1,\ldots,\mu_n)$ the diagonal entries of $\theta_\e$ with respect to $\omega$,  and reformulate the twisted J-flow as in Section \ref{section:twisted_j_flow}:
\[ \dot \f_\e =c_\e-\sum_{j=1}^n\frac{\mu_j}{\lambda_j}-\frac{\beta}{\lambda_1\ldots\lambda_n}.\]

We first  remark that 
\begin{lemma}\label{lem:vanishing_int}
 Along the twisted J-flow we have 
 \begin{equation}\label{eq:int_vanish}
     \sum_{j=0}^n\int_X\f_\epsilon \omega^j_{\f_\epsilon} \wedge \omega^{k-j}=0,
 \end{equation}
hence 
\begin{equation}\label{eq:sup_inf}
    \sup_X \f_\epsilon \geq 0 \text{ and } \inf_X \f_\epsilon\leq 0.
\end{equation}
\end{lemma}

\begin{proof}
Recall that the functional $I_\omega$  is defined by 
$$I_\om(\phi)=\int_0^1\int_X \frac{\partial \phi_t}{\partial t}\frac{\om_{\phi_{t}}^n}{n!}dt ,$$
for $\{\phi_t\}_{t\in [0,1]}$ a path in between $0$ and $\phi$. Observe that  $I_\om(\f_t)=0$ along the twisted J-flow. On the other hand, we also have (see for example \cite[Lemma 3.2]{Wein06})
\begin{equation}
    I_{\om}(\f) =  \frac{1}{(n+1)!} \sum_{k=0}^n\int_X \f\om^k\wedge\om_\f^{n-k},
\end{equation}
so we obtain \eqref{eq:int_vanish}.
\end{proof}
The following lemma uses the positivity condition \eqref{eq:pos_cond_e}. 
\begin{lemma}\label{lem:key_C0}
Let  $\delta_0<1, C_0$ and $c$ be positive constants. Assume that $ \mu=(\mu_1,\ldots,\mu_n)\in \mathbb{R}^n$ such that $ 0\leq \mu_i \leq C_0, \forall i=1,\ldots n$ and   $c -\sum_{j\neq k} \mu_j \geq \delta_0, \forall k=1,\ldots,n$. 

Assume that $\lambda = (\lambda_1,\ldots,  \lambda_n)\in \Gamma_n$ and $\tau\in \mathbb{R}$ satisfy $\lambda_k\geq 1-\delta,   \forall k>0$, $\tau\geq -\delta $ with $\delta=\delta_0/(4c+4)$  and  
 $$\sum _{k=1}^n\dfrac{\mu_k}{\lambda_k}+f(\lambda_1,\ldots,\lambda_n) =(c+\tau),$$ for some  continuous  function $f$ with $\lim_{\lambda_j\rightarrow +\infty} f(\lambda)=0, \forall i=1,\ldots ,n$. 
 
 Then there exists a uniform constant $K>0$ depending only on $\delta_0, n,c$, such that $$|\tau|\leq K\quad \text{and} \quad  \lambda_k\leq K, \forall k=1,\ldots,n.$$  
\end{lemma}
\begin{proof}
Observe that $$-c+ \inf_{S} f \leq \tau\leq \frac{n(c-\delta_0)}{(1-\delta)} +\sup_{S}f,$$
where $$ S= \{\lambda\in \Gamma_n: \lambda_j\geq 1-\delta, \forall j=1,\ldots,n\}.$$

Assume by contradiction that there is a sequence of $\lambda^{(\ell)} \in \Gamma_n$ and $k\in \{1,\ldots, n\}$ such that $ \lim_{\ell \rightarrow +\infty}\lambda^{(\ell)}_k=+\infty$, then there is  $\ell_0>0$ such that 
$$(c+\tau)- \frac{\delta_0}{4} \leq \sum_{j\neq k} \dfrac{\mu_j}{\lambda^{(\ell_0)}_j} \leq \dfrac{1}{1-\delta}\sum_{j\neq k} \mu_j \leq\frac{c-\delta_0}{1-\delta} .$$
Therefore 
$(1-\delta)[(c+\tau)-\delta_0/4] \leq c-\delta_0$, so 
$$\delta_0\leq  c \delta + (1-\delta)(\frac{\delta_0}{4} -\tau) \leq  c \delta + (1-\delta)(\frac{\delta_0}{4}  +\delta) \leq  \frac{\delta_0}{4}  +  \delta (c+1) =\frac{\delta_0}{2} ,$$ so we get a  contradiction. 
\end{proof}
We now deduce a  uniform lower bound for $\f_\e$ following an argument in \cite{PT}.

\begin{lemma}\label{lem:lower_bd}
There exists a uniform constant $C$ such that $\f_\epsilon\geq - C$. 
\end{lemma}
\begin{proof}
Since $\p_t \f_\epsilon$ is uniformly bounded for all time by the constant depending only on $\f_0$, we can consider $t\geq \delta$.  For any $T>\delta$, for each $t$ we note
$$L_\epsilon= \min_{X\times [0,T]} \f_\epsilon =\f_{\epsilon}(x_\epsilon,t_\epsilon).$$ 

Let $(z_1,\ldots,z_n)$ be a local coordinates centered at $x_\epsilon$, and $U=\{z: |z|<1 \}$  such that $ A^{-1}\om_{E}\leq \omega \leq A\omega_E$ where $\om_E $ is the Euclidean metric and $A$ is some uniform constant. Denote $\mathcal{U}_\epsilon= U\times\{t: -\delta\leq 2(t-t_\epsilon)<\delta\}$, and define
$$w_\epsilon= \f_\epsilon+\dfrac{\delta^2}{4}|z|^2 +|t-t_\epsilon|^2,$$
where $\delta>0$. 

Then $w_\epsilon$ attains its minimum on $\mathcal{U}_\epsilon$ at $(z_\epsilon,t_\epsilon)$ and $w_\e\geq \min _{\mathcal{U}_\e} w_\e +\frac{1}{\delta^2}$ on the boundary of $\mathcal{U}_\epsilon$. Now by the ABP inequality due to Tso \cite[Proposition 2.1]{Tso}, there exists a constant $C_n=C(n)>0$ so that 
$$C_n\delta^{4n+2}\leq \int_S(-\p_t w_\e)\det((w_\e)_{ij})dxdt,$$
where \[
S:=\left\lbrace (x,t) \in \mathcal{U}: \begin{array}{l}
  w_\e(x,t)\leq w_\e(z_0,t_0)+\frac{\delta^2}{4},\quad |D_x w_\e(x,t) |<\frac{\delta^2}{8}, \textit{ and }\\
  w_\e(y,s)\geq w_\e(x,t)+D_x w_\e(x,t). (y-x),\, \forall y\in U, s\leq t 
 \end{array}
\right\rbrace.
\]
Since on $S$, we have $D^2 w_\epsilon\geq 0$ and $\p _t w_\epsilon\leq 0$, then $\lambda_\omega(dd^c \f_\epsilon) \geq -\delta \mathbf{I}$, and  $0\leq -\p_t w \leq -\p_t u+\delta $, hence
\begin{equation}\label{eq:ineq_lem}
    \lambda_\omega(\omega+ dd^c u ) \geq (1-\delta)\mathbf{I} \text{ and } -\p_t \f_\epsilon \geq -\delta.
\end{equation}
 Since $\f_\e $ satisfies the equation \eqref{eq:J-flow_e}, we have
 \begin{equation}\label{eq:identity}
     \sum_{j=1}^n\dfrac{\mu_k  }{\lambda_k} +\frac{\beta}{\lambda_1\ldots \lambda_n}=(c_\e+\p_t \f_\epsilon),
 \end{equation}
where $\lambda=\lambda_\omega(\omega+dd^c \f_\e)$ and $\mu_1\ldots,\mu_n $  are diagonal entries of $\theta_\e$ with respect to normal coordinates of $\omega$.   It follows from Lemma \ref{lem:key_C0} that there exists a uniform constant $C$ such that $$|\dot w_\e|+\det(D^2(w_\e)_{jk})\leq C,$$
therefore
$$C_n\delta^{4n+2}\leq C'\int_Sdxdt $$
for a uniform constant $C'>0$.

Now on $S$ we have $w_\epsilon\leq L_\epsilon+ \frac{\delta^2}{4}$. Since we can assume that $|L_\epsilon|>\delta^2$ so that $L_\epsilon+ \delta^2/4<0$ and $|w_\e| \geq |L_\e|/2$, then we get

\begin{equation}\label{eq:ineq_lem_2}
    C_n\delta^{4n+2}\leq C' \int_Sdxdt \leq \frac{2    C'}{|L_\epsilon|}\int_S |w_\e|dxdt\leq  \frac{2C'    }{|L_\epsilon|}\int_\mathcal{U} |w_\epsilon|dxdt.
\end{equation}

Next we have 
$|w_\epsilon|=-w_\epsilon= -\f_\epsilon -\delta^2/4 |z|^2 -(t-t_0)^2\leq -\f_\epsilon\leq -\f_\epsilon +\sup_X\f_\epsilon,$ since $\sup_X\f_\epsilon\geq 0 $ by \eqref{eq:sup_inf}.

Since $\f_\epsilon$ is   $\omega$-psh, we have 
$\|\f_\epsilon-\sup_X\f_\e\|_{L^1(X,\omega^n)}\leq C $ for $C$ depending only on $(X,\omega)$ (cf. \cite[Proposition 8.5]{GZ}). Combining this with \eqref{eq:ineq_lem_2} yields
$$C_n\delta^{4n+2 }\leq \dfrac{2C'}{|L_\epsilon|} \int_{|t|<\frac{\delta}{2}}\|\f_\epsilon-\sup_X\f_\e\|_{L^1(X,\omega^n)}dt \leq C''\delta \dfrac{2}{|L_\epsilon|},$$
for a uniform constant $C''$.
This implies $\f_\epsilon \geq -C$ for a uniform constant $C$. 
\end{proof}

The uniform bound now follows from a Harnack type inequality for $\f_\e$.  
\begin{lemma}\label{lem:C0}
For any $T>0$
There exists a uniform constant $C$ such that 
$$\|\f_\epsilon\|_{C^0(X\times[0,T])}\leq C.$$
\end{lemma}
\begin{proof}
Along the twisted J-flow we have $\J^{\theta_\e}_{\om,\beta}$ is decreasing, hence combining with Lemma \ref{lem:vanishing_int} implies
$$\sum_{j=0}^{n-1}\int_X\f_\epsilon \omega^j_{\f_\epsilon} \wedge \omega^{n-1-j} \wedge \theta_\e\leq C_0,$$
where $C_0$ is a uniform constant.
Therefore
\begin{eqnarray*}
     \int_X \f_\e  \om^{n-1} \wedge \theta_\e
     &\leq&C_0 -\sum_{j=1}^{n-1} \int_X \f_\epsilon \omega^j_{\f_\epsilon} \wedge \omega^{n-1-j} \wedge \theta_\e \\
     & = & C_0-\sum_{j=1}^{n-1} \int_X (\f_\epsilon-\inf_X\f_\e) \omega^j_{\f_\epsilon} \wedge \omega^{n-1-j} \wedge \theta_\e -(n-1)\inf_X\f_\e \int_X  \omega^{n-1} \wedge \theta_\e\\
     &\leq & C_0-C\inf_X\f_\e,
\end{eqnarray*}
where $C=\max_{\e\in [0,\e_0]}(n-1) \int_X\om^{n-1} \wedge \theta_\epsilon$. 
So we have
$$\int_X\f_\epsilon \om^{n-1} \wedge \theta\leq  \int_X\f_\epsilon \om^{n-1} \wedge \theta_\e\leq C_0 -C\inf _X \f_\epsilon.$$
Since $\f_\e$ is $\omega$-psh, and $\mu:=\om^{n-1} \wedge \theta/ \int_X \om^{n-1} \wedge \theta$  is a probability measure, then by  compactness properties of $\omega$-psh function (cf. \cite[Prop. 8.5]{GZ}) there is a constant $C_\mu>0$ depending on $\mu$ such that $$\sup_X\f_\epsilon \leq \int_X \f_\epsilon \mu+C_\mu,$$
hence
$$\sup_X\f_\epsilon \leq C_\mu+\frac{C_0}{\int_X \om^{n-1}\wedge \theta} -\frac{C}{\int_X \om^{n-1}\wedge \theta}\inf \f_\e.$$
Now combining with Lemma \ref{lem:lower_bd}, we obtain the desired estimate.
\end{proof}

\subsection{Alternative proof of $C^0$ estimate}
We give here a second proof of the $C^0$ estimate to $\f_\e$ by proving a uniform  $C^0$ estimate for the twisted inverse $\sigma_{n-1}$  equation where the uniform bound does not depend on the norms with respect to $\theta$.

\begin{proposition}\label{prop:C0_elliptic}
Let $\psi, \rho\in C^\infty (X)$ be  smooth  functions on $X$ with $\psi>0$. 
Assume that there is $\hat\omega=\om+\dd \underline{u}\in [\omega]$ such that 
\begin{equation}\label{C0:eq_cond_pos}
    (\psi\hat\omega-(n-1)\theta)\wedge \hat\omega^{n-2}>0.
\end{equation}

Let $u$ be a smooth solution
of the equation \begin{equation}\label{eq:j_eq_1}
   \rho\om^n+ n\omega^{n-1}_u\wedge\theta=\psi \om _u^{n}, \quad \sup_X (u-\underline{u})=0,
    \end{equation}
where $\omega$ and $\theta$ are K\"ahler metrics. 
There exists a constant $C>0$ which only depends  on $X,\omega, \|\theta\|_{C^0(X,\om)}, \|\psi\|_{L^\infty}, \|\rho\|_{L^\infty}$ such that $$\|u\|_{L^\infty(X)}\leq C.$$ 
\end{proposition}

\begin{proof}
The proof uses an ABP estimate, as in  \cite{Sze}. Our $C^0$ bound only depends  on the norms with respect to $\omega$ while it depends on the  norms with respect to  $\theta$  in \cite{Sze}.   We can assume with out loss of generality that $\underline{u}=0$ by changing $\omega$. Then $\max_X u=0$, it suffices to get a lower bound for $L=\min_M u=u(x_0)$.
 
 Let $(z_1,\ldots, z_n)$ be a local coordinates centered at $x_0$, $U=\{z: |z|\leq 1\}$, such that $ A^{-1}\om_{E}\leq \omega \leq A\omega_E$ where $\om_E $ is the Euclidean metric and $A$ is some uniform constant.  Let $w= u+\delta|z|^2$. Then $\inf_{U}w = L= w( 0)$ and $w(z)\geq L+\de $ for $z\in \partial U$. It follows from the ABP maximum principle \cite[Lemma 9.2]{GT} that there exists a constant $C=C(n)$ so that

\[C(n)\delta^{2n}\leq \int_ S \det(D^2w), \]
  where $S$ is defined by
  
  \begin{equation}
      S=\left\lbrace x\in U: |D w|<\frac{\de}{2} \, \text{ and }  \, 
      w(y)\geq w(x)+D w(x).(y-x), \forall y\in U.
      \right\rbrace
  \end{equation}
 Since on $S$ we have $D^2 w\geq 0$, we infer $(u_{j\bar k })\geq -\delta I_{n}$, hence
  $ \lambda_\om(\om+\dd u)\geq (1-\de){\bf I}$. Now $u$ satisfies the equation \eqref{eq:j_eq_1}, hence
  \begin{equation}
    \frac{\rho}{\lambda_1\ldots\lambda_n} +\sum_{k=1}^n \frac{\mu_k}{\lambda_k} =\psi.
  \end{equation}
  As in Lemma \ref{lem:lower_bd}, the condition \eqref{C0:eq_cond_pos} implies that 
  \[ \psi - \sum_{j\neq k} \frac{\mu_k}{\lambda_k} \geq \delta_0 .\]
  It follows from Lemma \ref{lem:key_C0}
that there exists a uniform constant $C$ such that  
$\det (D^2 w_{jk})\leq C,$
hence $C(n)\delta^{2n} \leq C\int_S dx.$

Now on $S$ we have $w\leq L+\delta$. Since we can assume  that $|L|> 2 \delta$ so that $L+\de<0$ and $w\leq L/2$.   This implies 

\begin{equation}\label{eq:control_vol}
C(n)\delta^{2n}\leq C_1\int_S dx\leq \frac{2C_1}{|L| }\int_S |w| dx.\end{equation}
Since $u$ is $\omega$-psh with $\sup_X u=0$, there exists a uniform constant $C_2=C(X,\omega)$ such that $\|u\|_{L^1(X)}\leq C_2,$ hence $\| w\|_{L^1(U)}\leq C_3$.  Combining with \eqref{eq:control_vol} implies 
\[C(n)\delta^{2n}\leq  \frac{2C_1}{|L|}\|w\|_{L^1(U)}\leq \frac{C_4}{|L|} .\]
Therefore $u\geq -C$ for some uniform constant $C$.  
\end{proof}
We  now ready to give another proof of Lemma \ref{lem:C0}.  
\begin{proof}
We use an argument similar to one given in \cite{SW12}. 
Suppose that $u_\e$
 is the solution to the equation
 \begin{equation}
     \beta\om^n+ n\om_{u_\e}^{n-1}\wedge \theta_\e = c_\e \om^{n}_{u_\e}, \quad 
 \end{equation}
 where the existence of unique solution proved in \cite{SW08}
 Now set $\psi_\e= \f_\e-u_\e$ and compute 
 \begin{eqnarray}
     \dfrac{d\psi_\e}{d t}=  \dfrac{d\f_\e}{d t}&=&\dfrac{n\omega_{u_\e }^{n-1}\wedge \theta_\epsilon }{\omega_{u_\e} ^n}+\beta\frac{\om^n}{\om^n_{u_\e}} -\dfrac{n\omega_{\f_\e }^{n-1}\wedge \theta_\epsilon }{\omega_\f ^n}- \beta\frac{\om^n}{\om^n_{\f_\e}}\\
     &=& \int_0^1\dfrac{d }{ds } \left(\dfrac{n\omega_{s}^{n-1}\wedge \theta_\epsilon }{\omega_s} +\beta\frac{\om^n}{\om_s^n}\right) ds,
 \end{eqnarray}
    where $\om_s:=s\om_{u_\e} +(1-s)\om_{\f_\e}$. We define
    $$ \eta_s^{k\bar \ell} = \omega_s^{k \bar j }\omega_s^{i \bar \ell }\theta_{i\bar j}+\beta\frac{\om^n }{\om_{s}^n}\om_{s}^{k \bar \ell },$$ which is positive definite. Then we have
    \begin{eqnarray}
         \dfrac{d \psi_\e}{ dt }
          &=& \int_0^1\dfrac{d }{d s}\left( \dfrac{n\omega_{s }^{n-1}\wedge \theta }{\omega_{s} ^n} + \beta\frac{\om^n}{\om_s ^n}\right) ds\\
         &=& \int_0^1\eta^{k\bar \ell }(\omega_{\f_\e}-\omega_{u_\e})ds\\
         &=& 
        \left( \int_0^1\eta_s^{k\bar \ell } ds\right) \p_k \p_{\bar \ell} \psi_\e.
    \end{eqnarray}
    Since $  \left( \int_0^1\eta_s^{k\bar \ell } ds\right) $ is positive definite tensor, the maximum principle implies that $\psi_\e$ is uniformly bounded by $\sup_X|\psi_\e|$
at $t=0$.     Combining with  Lemma \ref{prop:C0_elliptic}, we have $\psi_\e$ is uniformly bounded. Therefore $\f_\e$ is uniformly bounded independent of $\e$. 
 \end{proof}

\subsection{$C^2$ estimate}
We now assume that $\theta$ also satisfies the condition \eqref{eq:cond_1}:  there is  an effective divisor $D$ on $X$,  such that
\begin{equation}
\theta\geq C_0 | s|_h^{2\gamma} \omega, \text{ and } \theta-\varepsilon_0 R_h\geq C_0 \omega,
\end{equation}
for some constants $\gamma,\varepsilon_0>0$,
where $h$ is a hermitian metric on $\mathcal{O}_X(D)$ and $s$ is a holomorphic section of $\mathcal{O}_X(D)$. 
Then we have the following $C^2$ estimate for $\f_\e$. 

\begin{lemma}\label{lem:C2}
There exist uniform positive constants $C,\alpha$, independent of $\epsilon$, such that $$\tr_{\omega} \omega_{\f_\e}\leq \dfrac{C}{|s|_{h}^{2\alpha}}.$$
\end{lemma}

\begin{proof} To simplify the notations, we drop all subscripts  $\e$, so we write $\varphi$ (resp.  $\theta$) instead of $\varphi_\epsilon$ (resp. $\theta_\epsilon$). The constant $C>0$  below will be independent of  $\epsilon$.

We use  Tsuji's trick \cite{Tsu}:  set  $\tilde\f = \f - \varepsilon \rho$, where $\rho =\log|s|^2_h$ for some small constant $\varepsilon$ that will be chosen below. 
Since $\f$ is uniformly bounded by Lemma \ref{lem:C0}, $\tilde\f $ is uniformly bounded from below and tends to $+\infty$ along $D$.   We set
$$H= \log \tr_{\omega} \omega_{\f} -  A\tilde\f,$$
where $A>0$ will be chosen hereafter. 
 It is straightforward that $H$ achieves a maximum at each time $t$ away from $D$.
We now prove that $H$ is bounded from above. We will use the maximum principle and follow the computation for the (twisted) J-flow  in \cite{SW08,SW12, Zhe15,Zhe18} to simplify  $(\partial_t-\mathcal{L})H$,
where $$\mathcal{L}= \left( \om_\f^{k\bar j} \om_\f^{i\bar \ell} \theta_{k\bar \ell}+\beta\frac{\om^n}{\om_\f^n}\om_\f^{i\bar j}\right) \p_i\p_{\bar j},$$
here we write $\omega= \sum_{k,\ell}i\om_{k\bar \ell}dz^k\wedge d\bar z^{\ell}$ in local coordinates and $(\om^{k\bar\ell})= (\om_{k\bar\ell})^{-1}$.

\medskip
 Set $\La=\tr_{\om}\om_\f$.  Using the flow equation we have
 \begin{eqnarray}
      \partial_t \La&=& \om^{i\bar j} \p_t\om_{\f, i\bar j} =\om^{i\bar j}\dot \f_{i\bar j}\\
      &=&\om^{i\bar j}\left[  c_\beta -\om_\f^{p\bar q}\theta_{p\bar q}-\beta \frac{\om^n}{\om^n_\f} \right]_{i\bar j}\\
      &=& \om^{i\bar j}[ \om_\f^{r\bar q}\om_\f^{p\bar s}(\om_{\f, r\bar s})_{i\bar j}-\om_\f^{r\bar q}\om_\f^{p\bar b}\om_\f^{a\bar s}(\om_{\f, a\bar b})_{\bar j}(\om_{\f, r\bar s})_{i} \\
      && -\om_\f^{r\bar b}\om_\f^{a\bar q}\om_\f^{p\bar s}(\om_{\f, a\bar b})_{\bar j}(\om_{\f, r\bar s})_{i}]\theta_{p\bar q} + 2Re( \omega^{i\bar j} \omega_\f^{p\bar \ell}\omega_\f^{k\bar q}(\omega_{\f, k\bar \ell})_{i} (\theta_{p\bar q })_{\bar j} )\\
      &&- \om^{i\bar j}\om_\f^{p\bar q}\theta_{p\bar q i\bar j}-\beta\om^{i\bar j} \left(\frac{\om^n }{\om_\f^n}\right)_{i\bar j}.
 \end{eqnarray}

\medskip
In the normal coordinates of $\om$ we have
\begin{eqnarray*}
     \left(\frac{\om^n }{\om_\f^n}\right)_{i\bar j}&=& -\om^{k\bar \ell}R_{k\bar \ell i\bar j}(\om) \frac{\om^n}{\om_\f^n}+\frac{\om^n}{\om_\f^n}\om_\f^{p\bar q} \om^{k \bar \ell }(\om_{\f, p\bar q})_{\bar j} (\om_{\f, k\bar \ell})_i\\
     &&+ \frac{\om^n}{\om_\f^n}\om_\f^{k\bar q} \om^{p \bar \ell }(\om_{\f, p\bar q})_{\bar j} (\om_{\f, k\bar \ell})_i-\frac{\om^n}{\om_\f^n} \om_\f^{k\bar \ell} (\om_{\f, k\bar \ell})_{i\bar j},
\end{eqnarray*}
 and
 $$\La_{k\bar \ell} = R^{p\bar q}{}_{k\bar \ell} (\om) \om_{\f, p\bar q} +\om^{p\bar q}(\om_{\f, p\bar q})_{k\bar \ell}.$$
 
Now, we have

\begin{eqnarray}\label{term_1}
      \left( \frac{\p }{\p t}-\mathcal{L} \right) H &=& \frac{1}{\La}\partial_t \La -\om_\f^{k\bar j}\om_\f^{i\bar \ell}\theta_{i\bar j} \left( \frac{\La_{k\bar \ell}}{\La}-\frac{\La_k\La_{\bar \ell}}{\La^2}\right)\\\label{term_2}
      &&-\beta  \frac{\om^n}{\om^n_\f}\om_\f^{k\bar \ell} \left( \frac{\La_{k\bar \ell}}{\La}-\frac{\La_k\La_{\bar \ell}}{\La^2}\right)\\\label{term_3}
      && -A(\partial_t-\mathcal{L})\tilde \f.
\end{eqnarray}
The terms in \eqref{term_1} and \eqref{term_2} give 
\begin{eqnarray}\label{row_4}
\frac{\p_t \Lambda - \om_\f^{k\bar j}\om_\f^{i\bar \ell} \theta_{i\bar j}\Lambda_{k\bar\ell}}{\Lambda} &+&\frac{\om_\f^{k\bar j}\om_\f^{i\bar \ell}\theta_{i\bar j} \Lambda_k\Lambda_{\bar \ell}}{\Lambda^2}\\\label{row_5}
-\beta\frac{\om^n}{\om_\f^n}\frac{\om_\f^{k\bar \ell}\Lambda_{k\bar \ell}}{\Lambda} &+&\beta\frac{\om^n}{\om_\f^n}\om_\f^{k\bar \ell}\frac{\Lambda_k\Lambda_{\bar \ell}}{\Lambda^2}.
\end{eqnarray}
Developing  \eqref{row_4} using previous calculations, we get 
\begin{eqnarray}
  \label{row_6}    &&\frac{1}{\Lambda}\left\{\om^{i\bar j} \om_\f^{r\bar q}\om_\f^{p\bar s} (\om_{\f, r\bar s})_{i\bar j}\theta_{p\bar q} -\om^{i\bar j} \om_\f^{r\bar q} \om_\f^{p\bar b}\om_\f^{a\bar s} (\om_{\f, a\bar b})_{\bar j} (\om_{\f, r\bar s})_{i} \theta_{p\bar q} \right.\\ &&\label{row_7}  
  -\om^{i\bar j} \om_\f^{r\bar b}\om_\f^{a\bar q} \om_\f^{p\bar s} (\om_{\f, a\bar b})_{\bar j} (\om_{\f, r\bar s})_{i} \theta_{p\bar q}  +
 2Re( \omega^{i\bar j} \omega_\f^{p\bar \ell}\omega_\f^{k\bar q}(\omega_{\f, k\bar \ell})_{i} (\theta_{p\bar q })_{\bar j} )    \\ \label{row_8} 
   &&  -\om^{i\bar j} \om_\f^{p\bar q}\theta_{p\bar qi\bar j}+\beta \om^{i\bar j} \om^{k\bar \ell}R_{k\bar \ell i\bar j}(\om) \frac{\om^n}{\om^n_\f} -\beta\frac{\om^n}{\om_\f^n}\om^{i\bar j}\om_\f^{k \bar \ell} \om_\f^{p \bar q}(\om_{\f, p\bar q})_{\bar j}(\om_{\f, k\bar \ell})_i\\  \label{row_9} 
   &&\left.  -\beta\frac{\om^n}{\om_\f^n}\om^{i\bar j}\om_\f^{k \bar q} \om_\f^{p \bar \ell}(\om_{\f, p\bar q})_{\bar j}(\om_{\f, k\bar \ell})_i +\beta \frac{\om^n}{\om_\f^n}\om^{i\bar j}\om_\f^{k \bar \ell} (\om_{\f, k\bar \ell})_{i\bar j}\right\} \\ \label{row_10}
   && -\frac{1}{\Lambda}\om_\f^{k\bar j}\om_\f^{i\bar \ell}\theta_{i\bar j}R^{p\bar q}{}_{k\bar \ell}(\om)\om_{\f, p\bar q}  -\frac{1}{\Lambda}\om_\f^{k\bar j}\om_\f^{i\bar \ell}\theta_{i\bar j}\om^{p\bar q}(\om_{\f, p\bar q})_{k\bar \ell}  +\frac{ \om_\f^{k\bar j}\om_\f^{i\bar \ell}\theta_{i\bar j}\Lambda _k\Lambda_\ell}{\Lambda^2}
\end{eqnarray}
Developing \eqref{row_5} we obtain
\begin{eqnarray}  \label{row_11} 
-\beta\frac{\om^n}{\om_\f^n} \frac{\om_\f^{k\bar \ell}R^{p\bar q}{}_{k\bar \ell}(\om)\om_{\f, p\bar q}}{\Lambda}-\beta\frac{\om^n}{\om_\f^n}\frac{\om_\f^{k\bar \ell}\om^{p\bar q} (\om_{\f, p\bar q})_{k\bar \ell}}{\Lambda}+
     \beta\frac{\om^n}{\om_\f^n}\om_\f^{k\bar \ell}\frac{\Lambda_k\Lambda_{\bar \ell}}{\Lambda^2}. 
\end{eqnarray}
Now we have two following inequalities due to Weinkove \cite{Wein04} and Zheng \cite[Lemma 7]{Zhe15}:

\begin{equation}
    [\om^{i\bar j}\om_{\f}^{r\bar b}\om_\f^{a \bar q } \om_\f^{p\bar s} (\om_{\f, a\bar b})_{\bar j}(\om_{\f, r\bar s})_i\theta_{p\bar q}]\Lambda\geq \om_\f^{k\bar j}\om_\f^{i\bar \ell}\theta_{i\bar j}\Lambda_{k}\Lambda_{\bar \ell}
\end{equation}

\begin{equation}
    [\om^{i\bar j}\om_{\f}^{k\bar q}\om_\f^{p\bar \ell}(\om_{\f, p\bar q})_{\bar j}(\om_{\f, k\bar \ell})_i]\Lambda\geq \om_\f^{k\bar \ell}\Lambda_{k}\Lambda_{\bar \ell}.
\end{equation}
Therefore the last term in \eqref{row_10} is dominated by  the first term in \eqref{row_7} and the last term in \eqref{row_11} is dominated with  the first term in \eqref{row_9}. In addition, the last term in \eqref{row_9}  is canceled by the second term in \eqref{row_11} and the second term in \eqref{row_10} is canceled by the first term in \eqref{row_6}. 


We now follow \cite{SW12} to deal with the term involving one derivative of $\omega_\f$ and $\theta$.  We have
$$G= \omega^{ i\bar j} \omega_\f^{p \bar \ell}  \tilde{\om}^{k\bar q} K_{i \bar \ell  k}\overline{K_{j \bar p q } } \geq 0,$$
where  $\tilde{\om}^{k\bar q}:= \omega_{\f}^{k\bar s} \omega_{\f}^{r\bar q} \theta_{r\bar s}$  and
$$K_{i\bar \ell k} = (\omega_{\f, k\bar \ell})_{i} -\theta ^{a \bar b} \omega_{\f, k\bar b } \theta_{a\bar \ell i }.$$
By a direct computation, we get
\begin{eqnarray*}
G= \omega^{i\bar j}\omega_{\f}^{p\bar \ell} \tilde{\om} ^{k\bar q} (\omega_{\f, k\bar \ell})_i (\omega_{\f, p\bar q})_{\bar j} -2Re(\omega^{i\bar j}\omega_{\f}^{p\bar \ell} \omega_{\f}^{k\bar b} (\omega_{\f, k\bar \ell })_{i } (\theta_{p\bar b})_{ \bar j}) + \omega^{i\bar j}\omega_{\f}^{p\bar \ell} \theta^{a\bar s}\theta_{a\bar \ell i}\theta_{p\bar s \bar j}\geq 0,
\end{eqnarray*} 
We infer that
\begin{eqnarray*}
- \omega^{i\bar j}\omega_{\f}^{p\bar \ell} \omega_{\f}^{k\bar s} \omega_{\f}^{r\bar q} \theta_{r\bar s} (\omega_{\f, k\bar \ell})_i (\omega_{\f, p\bar q})_{\bar j} +2Re(\omega^{i\bar j}\omega_{\f}^{p\bar \ell} \omega_{\f}^{k\bar b} (\omega_{\f, k\bar \ell })_{i } (\theta_{p\bar b})_{ \bar j}) &\leq&  \omega^{i\bar j}\omega_{\f}^{p\bar \ell} \theta^{a\bar s}\theta_{a\bar \ell i}\theta_{p\bar s \bar j},
\end{eqnarray*} 
hence  the second term in \eqref{row_7}  minus the second term in \eqref{row_6}  is dominated by   
$\omega^{i\bar j}\omega_{\f}^{p\bar \ell} \theta^{a\bar s}\theta_{a\bar \ell i}\theta_{p\bar s \bar j}.$ 


Finally, we infer that the  terms in \eqref{term_1} and \eqref{term_2} are dominated by
\begin{eqnarray*}
   &&  \frac{1}{\La} \left\{ -\om_\f^{p\bar q}\theta_{p\bar q i \bar j} \om^{i\bar j}  - \om_\f^{k\bar j}\om_\f^{i\bar \ell}\theta_{i\bar j}R^{p\bar q}{}_{k\bar \ell} (\om)\om_{\f, p\bar q} +\omega^{i\bar j}\omega_{\f}^{p\bar \ell} \theta^{a\bar s}\theta_{a\bar \ell i}\theta_{p\bar s \bar j} \right.   \\
   &&  \left.  \quad +\beta \om^{i\bar j}\om^{k\ell}R_{k\bar \ell i\bar j }(\om) \frac{\om^n}{\om_\f^n} -\beta\frac{\om^n  }{\om_\f^n}\om_\f^{k\bar \ell}R^{p\bar q}{}_{k\bar \ell}(\om)\om_{\f, p\bar q}       \right\}.
\end{eqnarray*}
Therefore at  a maximum point $(t_0,x_0) $ of $H$, 
\begin{eqnarray}\nonumber
    0\leq  \left( \frac{\p }{\p t}-\mathcal{L} \right) H &\leq&\frac{1}{\tr_\om\om_\f} \left\{- \om_\f^{p\bar q}\theta_{p\bar q i \bar j} \om^{i\bar j} +\beta \om^{i\bar j}\om^{k\ell}R_{k\bar \ell i\bar j }(\om) \frac{\om^n}{\om_\f^n}  + \omega^{i\bar j}\omega_{\f}^{p\bar \ell} \theta^{a\bar s}\theta_{a\bar \ell i}\theta_{p\bar s \bar j}  \right\} \\ \nonumber
    & &-\frac{1}{\tr_\om\om_\f} \left\{\om_\f^{k\bar j}\om_\f^{i\bar \ell}\theta_{i\bar j}R^{p\bar q}{}_{k\bar \ell} (\om)\om_{\f, p\bar q} +\beta\frac{\om^n  }{\om_\f^n}\om_\f^{k\bar \ell}R^{p\bar q}{}_{k\bar \ell}(\om)\om_{\f, p\bar q} \right\}\\ \label{inq:max_principle}
&&- A \left( \frac{\p }{\p t}-\mathcal{L}\right) \tilde \f.
\end{eqnarray}
Since $\om_\f\geq C\theta\geq|s|^{2\gamma}_h\om$ by Lemma \ref{lem:c1_t}, we have
\[\frac{1 }{\tr_\om\om_\f}\om_{\f}^{p\bar q}(-\theta_{p\bar q i\bar j})\om^{i\bar j}\leq \frac{C\tr_{\om_\f}\om }{\tr_\om\om_\f} \leq \frac{1}{|s|^{2\gamma}_h} \frac{C}{\tr_\om\om_\f}.\]
and 
\[\frac{1 }{\tr_\om\om_\f}\omega^{i\bar j}\omega_{\f}^{p\bar \ell} \theta^{a\bar s}\theta_{a\bar \ell i}\theta_{p\bar s \bar j} \leq  \frac{C\tr_{\omega_\f} \omega \tr_{\theta } \omega }{\tr_\om\om_\f} \leq \frac{C \tr_{\om_\f} \om \tr_{\om}\om  }{\tr_\om \om_\f} \leq \frac{1}{|s|^{4\gamma}_h} \frac{C}{\tr_\om\om_\f}. \]
 At  $(t_0,x_0)$, we can assume that  $$\frac{1}{|s|^{4\gamma}_h} \frac{C}{\tr_\om\om_\f}\leq 1.$$ Otherwise $\tr_{\om}\om_\f\leq \frac{C}{|s|^{4\gamma}_h}$ at $(t_0,x_0)$, then by choosing $A=4\gamma/\delta$ we get the desired estimate. We also have
$$\beta\frac{1 }{\tr_\om\om_\f} \om^{i\bar j}\om^{k\bar \ell}R_{k\bar \ell i\bar j }(\om) \frac{\om^n}{\om_\f^n}\leq    \beta  \frac{C}{\tr_\om \om_\f} \frac{\omega^n}{\om_\f^n}= \frac{C }{\tr_\om\om_\f}  (c_\beta-\dot \f-\tr_{\om_\f}\theta)\leq \frac{C }{\tr_\om\om_\f}, $$
since $\dot \f$ is uniformly bounded by Lemma \ref{lem:c1_t}. 
 At  $(t_0,x_0)$, we can also assume that 
 $$\frac{C }{\tr_\om\om_\f}\leq 1, $$
 otherwise we have desired estimate. 
 
The first term in the second line of \eqref{inq:max_principle} is controlled as
\begin{equation}
-\frac{\om_\f^{k\bar j}\om_\f^{i\bar \ell}\theta_{i\bar j}R^{p\bar q}{}_{k\bar \ell} (\om)\om_{\f, p\bar q} }{ \tr_{\omega}\omega_{\f}}\leq \frac{C \om_\f^{k\bar j}\om_\f^{i\bar \ell}\theta_{i\bar j} \omega_{k\bar \ell} \tr_{\omega}\omega_{\f}}{\tr_{\omega}\omega_{\f}}= C \om_\f^{k\bar j}\om_\f^{i\bar \ell}\theta_{i\bar j} \omega_{k\bar \ell}.
\end{equation}

 For the last term in the second line of \eqref{inq:max_principle},
 we have
 
 \begin{equation}\label{inq:3_term}
    - \frac{1 }{\tr_\om\om_\f}  \beta\frac{\om^n  }{\om_\f^n}\om_\f^{k\bar \ell}R^{p\bar q}{}_{k\bar \ell}(\om)\om_{\f, p\bar q} \leq C \beta\frac{\om^n}{\om_\f^n}\frac{\tr_{\om_\f}\om\tr_{\om}\om_\f}{\tr_\om\om_\f}= C \beta\frac{\om^n}{\om_\f^n}\tr_{\om_\f}\om.
 \end{equation}

\medskip
Thus at a maximum point $(t_0,x_0)$ of $H$, 
\begin{equation}\label{ineq:H_est}
    0\leq  \left( \frac{\p }{\p t}-\mathcal{L} \right) H\leq C \left( \beta\frac{\om^n}{\om_\f^n}\tr_{\om_\f}\om +  \om^{k\bar j}\om_\f^{i\bar \ell}\theta_{i\bar j} \omega_{k\bar \ell} \right)+3 - A \left( \frac{\p }{\p t}-\mathcal{L}\right) \tilde \f. 
\end{equation}
By the condition \eqref{eq:cond_1} there exists $\varepsilon''>0$ satisfies $\om-\varepsilon'' R_h\geq C_0\om$ for some $C_0>0$. Hence we can write $\varepsilon $ as $\varepsilon=\varepsilon'\varepsilon''$. Denote by $\tilde \om$ a positive $(1,1)$-form defined by $\tilde\om^{i\bar j}=\om_\f^{k\bar j} \om_\f^{i\bar \ell} \theta_{k\bar \ell}$, then we get 
\begin{eqnarray}
     \mathcal{L}(\varepsilon \rho)&= &\Delta_{\tilde \om}(\varepsilon \rho) +\beta\frac{\om^n}{\om_\f ^n}\Delta_{\om_\f} (\varepsilon\rho), \quad (\Delta_{\om}:=\tr_\om \dd)\\
     &=&  \varepsilon'\tr_{\tilde \om } (\om-\varepsilon''R_h) +\varepsilon' \beta\frac{\om^n}{\om_\f^n} \tr_{\om_\f}  (\om-\varepsilon''R_h) - \varepsilon'(\tr_{\tilde\om} \om+\beta\frac{\om^n}{\om_\f^n}\tr_{\om_\f}\om)\\
     &\geq &\varepsilon'C_0\tr_{\tilde\om}\om +\varepsilon'C_0\beta\frac{\om^n}{\om_\f^n}\tr_{\om_\f} \om -\varepsilon'(\tr_{\tilde\om} \om+\beta\frac{\om^n}{\om_\f^n}\tr_{\om_\f}\om).
\end{eqnarray}
Combining with
\begin{eqnarray}
     (\partial_t-\mathcal{L})\f= -c_\beta -2\tr_{\om_\f}\theta +\om_{\f}^{k \bar j}\om_{\f}^{i\bar \ell}\theta_{i\bar j} \om_{k\bar \ell}-\beta(n+1)\frac{\om^n}{\om^n_\f} +\beta\frac{\om^n}{\om_\f^n}\om_\f^{k\bar \ell}\om_{k\bar \ell}
\end{eqnarray}
yields

\begin{eqnarray}\nonumber
     -A(\partial_t-\mathcal{L})\tilde\f&=& -A[-c_\beta -2\tr_{\om_\f}\theta +\om_{\f}^{k \bar j}\om_{\f}^{i\bar \ell}\theta_{i\bar j} \om_{k\bar \ell}-\beta(n+1)\frac{\om^n}{\om^n_\f} +\beta\frac{\om^n}{\om_\f^n}\om_\f^{k\bar \ell}\om_{k\bar \ell}]\\\label{ineq:L_t_phi}
    && -A[\varepsilon'C_0\tr_{\tilde\om}\om +\varepsilon'C_0\beta\frac{\om^n}{\om_\f^n}\tr_{\om_\f} \om -\varepsilon'(\tr_{\tilde\om} \om+\beta\frac{\om^n}{\om_\f^n}\tr_{\om_\f}\om)]
\end{eqnarray}
Choose $A$ large such that 
$$ A \varepsilon'C_0 \geq C,$$
so that  $$C\beta\frac{\om^n}{\om_\f^n}\tr_{\om_\f}\om -A\varepsilon'C_0\beta\frac{\om^n}{\om_\f^n}\tr_{\om_\f}\om\leq 0 \text{ and }  C \om_\f^{k\bar j}\om_\f^{i\bar \ell}\theta_{i\bar j} \omega_{k\bar \ell}- A \varepsilon'C_0 \tr_{\tilde{\omega}}\om \leq 0, $$
where $\tilde\om^{k \bar \ell}=\om_\f^{k\bar j}\om_\f^{i\bar \ell}\theta_{i\bar j}$ and $ C$ is the  constant in \eqref{ineq:H_est}.
Then we infer from \eqref{ineq:H_est} and \eqref{ineq:L_t_phi} that
\begin{eqnarray}
     0&\leq &\nonumber C_1-A\left\{c_\beta-2\tr_{\omega_{\f}} \theta+ \om_\f^{k\bar j} \om_\f^{i\bar \ell}\theta_{i\bar j }\om_{k\bar \ell} - \beta(n+1)\frac{\om^n}{\om_\f^n} -\varepsilon'\tr_{\tilde\om}\om \right\}\\
       &=  &C_1-A\left\{c_\beta-2\tr_{\omega_{\f}} \theta+(1-\varepsilon') \tr_{\tilde\om}\om - \beta(n+1)\frac{\om^n}{\om_\f^n} \right\}.\label{ineq_final}
\end{eqnarray}
We now consider  two cases.

\medskip
\noindent
{\bf Case 1:} assume that 
\begin{equation}
    c_\beta-2\tr_{\omega_{\f}} \theta+(1-\varepsilon')\tr_{\tilde \om}\omega  \leq \delta_1=\frac{\delta_0}{4}
\end{equation}
at a maximum point of $H$, where $\delta_0$ is the constant in the condition \eqref{eq:cond_2_e}.

Choosing normal coordinates for $\omega$ such that $\omega_{\f}$ is diagonal with entries $\lambda_1,\ldots, \lambda_n$ and $\theta$ may not be diagonal but we denote its positive diagonal entries by $\mu_1, \ldots \mu_n$. 
Then the above inequality becomes
$$ c_\beta- 2 \sum_{j=1}^n \frac{\mu_j}{\lambda_j} +\sum_{j=1}^n \frac{(1-\varepsilon')}{\lambda_j^2}\mu_j\leq \delta_1.$$
 Now by  condition \eqref{eq:pos_cond_e}, we have
$$(c_\beta\omega-(n-1)\theta)\wedge\omega^{n-2}\wedge \beta_k> \delta_0 \omega^{n-1}\wedge \beta_k,$$
where $\beta_k= idz^{k}\wedge\bar z^{ k}$
 In our coordinates, we get
 $$c_\beta(n-1)!\beta_1\wedge\ldots\wedge \beta_n - (n-1)!\sum_{j\neq k} \mu_j\beta_1\wedge\ldots\wedge \beta_n> \delta_0 (n-1)!\beta_1\wedge\ldots\wedge \beta_n,$$
 so
 $$c_\beta-\sum_{j\neq k} \mu_j>\delta_0,\, \forall k=1,\ldots n.$$
 
Therefore for any $k=0,\ldots,n$ we have
\begin{eqnarray} \nonumber
     \delta_1&\geq& c_\beta + \sum_{j\neq k}\mu_j\left( \frac{ \sqrt{ 1-\varepsilon'}}{\lambda_j} - \frac{1}{\sqrt{ 1-\varepsilon'}} \right)^2 -\frac{1}{ 1-\varepsilon'}\sum_{j\neq k}\mu_j+\frac{\mu_k}{\lambda_k^2}(1-\varepsilon')-2\frac{\mu_k}{\lambda_k}\\
     &\geq& c_\beta-\frac{1}{ 1-\varepsilon'}(c_\beta-\delta_0) -2\frac{\mu_k}{\lambda_k}.
\end{eqnarray}
 Hence 
 \begin{equation}
    2 \frac{\mu_k}{\lambda_k}\geq - \delta_1 + \frac{\delta_0-c_\beta \varepsilon'}{1-\varepsilon'}\geq \frac{\delta_0}{10},\forall k=1, \ldots,n,
 \end{equation}
 where we choose $ \varepsilon$ small enough such that $ c_\beta \varepsilon'\leq \delta_0/4$.
 This implies that
 $\tr_{\theta}\omega_{\f}\leq C, $
 hence 
 $$\tr_{\omega}\omega_{\f}\leq C,$$
 at a maximum point of $H$
 since $\theta\leq C'\omega$. 
 Therefore $H$ is bounded from above, and we get the desired estimate for $\tr_{\omega}\omega_{\f}$. 
 
 \medskip
 \noindent
{\bf Case 2:} assume that at a maximum point of $H$
 \begin{equation}
    c-2\tr_{\omega_{\f}} \theta+(1-\varepsilon')\tr_{\tilde \om}\omega  > \delta_1,
\end{equation} 
then \eqref{ineq_final} yields
$$\beta(n+1)\frac{\om^n}{\om_\f^n} \geq \delta_1- \frac{C_1}{A}.$$
Choosing $A$ large enough we get $$\beta(n+1)\frac{\om^n}{\om_\f^n} \geq\delta_1/2.$$
By the equation we have
$\tr_{\om_\f}\theta\leq c_\beta-\dot \f-\delta_1/2 \leq C$. Therefore

\begin{eqnarray}
     \tr_{\om}\om_\f&\leq& \frac{\om_\f^n}{\om^n}(\tr_{\om_\f}\om)^{n-1}\\
     &\leq & C\frac{1}{|s|^{2(n-1)\gamma}},
\end{eqnarray}
here we use the assuption $\theta\geq |s|_h^{2\gamma}\om$.

This implies that 
$H\leq H(t_0,x_0)\leq C-(n-1)\gamma\log |s|^2_h(x_0)+ A\delta \log |s|^2_h(x_0)\leq C$,
since we can choose $A$ is big enough such that $A\delta> (n-1)\gamma$.  We thus get the inequality as required.
\end{proof} 
Higher order estimates follow by  standard local parabolic theory.
We summarize this section by the following proposition.
 \begin{proposition}\label{prop:estimates}
 There exists a uniform constant $C$ such that for all $t$ and $\e\in (0,\e_0]$, 
 \begin{equation}
     \|\f_\e\|_{L^\infty(X)}\leq C \text{ and } \| \dot \f_\e\|_{L^\infty(X)}\| \leq C,
 \end{equation}
 Moreover, for any compact set $K\subset X\setminus D$ and $k\geq 0$, the exists $C_{K,k}>0$ such that 
\begin{equation}
    \| \f_\e\|_{C^k(K,\om)} \leq C_{K,k}.
\end{equation}
 \end{proposition}
 
 \section{Proof of the main result}
 The rest of  proof of Theorem \ref{thm:1} is similar to that of \cite[Theorem 1.1]{SW12}.
 \begin{proof}[Proof of Theorem \ref{thm:1}]
  By the estimates in Section \ref{sect:est} (Proposition \ref{prop:estimates}), we can find a sequence $\e_j\rightarrow 0$ such that $\f_{\e_j}$  converges to $\f$ in $C^\infty$ on compact sets of $(X\setminus D)\times [0,+\infty)$. Therefore $\f$ satisfies the degenerate twisted J-flow.  Since $\omega+\dd \f>0$ on $X\setminus D$ and $\sup_{X\setminus D}|\f|\leq C$ by Lemma \ref{lem:C0}, we can extend $\f$ uniquely to bounded $\omega$-psh function on $X$,  still denoted by $\f$.  

\medskip
The  $\J$-functional
\begin{eqnarray} \nonumber
     \J^\theta_{\omega,\beta}(\f)&=&\frac{1}{n!}\int_X \f \sum_{k=0}^{n-1}\theta \wedge \omega^k\wedge\omega_{\f}^{n-1-k}  -\frac{1}{(n+1)!}\int_X \underline{\theta} \f \sum_{k=0}^n \omega^k\wedge \omega^{n-k}_\f \\ 
     &&+\beta \int_X \f \frac{\omega^n}{n!} -\beta \frac{1}{(n+1)!}\int_X  \f \sum_{k=0}^n \omega^k\wedge \omega^{n-k}_\f,
\end{eqnarray}
where 
$$\underline{\theta}=\dfrac{n[\theta]\cdot [\omega]^{n-1}}{[\omega]^n},$$ 
 is well defined on $PSH(X,\omega)\cap L^\infty(X)$ with $\J^{\theta}_{\om, \beta
 }(\f+ C)= \J^{\theta}_{\om, \beta
 }(\f), \forall \f\in PSH(X,\omega)\cap L^\infty(X) $ and for any constant $C$. 
The uniform $L^\infty$ bound  on $\f$ yields 
\begin{equation}
    \J^\theta_{\omega, \beta}(\f(t))\geq -C,
\end{equation}
for a uniform constant $C$ independent of $t$. 
We also have 
\begin{equation}\label{J_decreasing}
    \frac{d}{dt}\J^\theta_{\omega, \theta}(\f(t))=  \frac{d}{dt}      \int_0^t \int_{X\setminus D} \dot \f(s)(n \theta\wedge \omega^{n-1}  + \beta\om^n-c_\beta \omega_{\f(s)}^n)  \frac{ds}{n!} = -\int_{X\setminus D} (\dot \f(s))^2 \dfrac{\omega_{\f(s)}^n}{n!}\leq 0,
\end{equation}
where $c_\beta=\underline{\theta}+\beta$. Therefore there exists a constant $C$ such that 
\begin{equation}
 \int_{X\setminus D} (\dot \f)^2 \dfrac{\omega_{\f}^n}{n!}\leq C.
\end{equation}
Then by an argument by contradiction in \cite[Proof of Theorem 1.1]{FLSW} shows that $ \dot \f(t)\rightarrow 0$ in $C^\infty_{loc}(X\setminus D)$. 
 Since we have uniform $C^\infty $ bounds for $\f$ on compact subsets of $X\setminus D$, the Arzel\`a-Ascoli theorem implies there is a sequence $t_j\rightarrow +\infty$, $\f(t_j)\rightarrow \f_\infty$ in $C_{loc}^{\infty}(X\setminus D)$ to a  $\f_\infty\in \mathcal{H}^{weak}_\om $. Since $\dot \f(t)\rightarrow 0$, $\f_\infty$ satisfies the equation 
 \begin{eqnarray}\label{eq:j_eq_lim}
      c_\beta\omega^n_{\f_\infty}= n\omega_{\f_\infty}^{n-1}\wedge \theta +  \beta \om^n.
 \end{eqnarray}

 We now prove the uniqueness of the solution $\f(t)$ to the degenerate J-flow.  Define
 $\phi_\delta = \f-\psi -\delta \rho $,  for some $\lambda>0$ will be chosen later,  $\omega_s= s\omega_\f+(1-s)\omega_{\psi}$ and $ \eta_s^{k\bar \ell} = \omega_s^{k \bar j }\omega_s^{i \bar \ell }\theta_{i\bar j}+\beta\frac{\om^n }{\om_{s}^n}\om_{s}^{k \bar \ell }$. Then on $X\setminus D$ we have
 \begin{eqnarray}
      \frac{\p \phi_\delta}{\p t} = \dot \f-\dot \psi&= & \left( \dfrac{n\omega_{\psi }^{n-1}\wedge \theta }{\omega_\psi ^n}  +\beta\frac{\om^n}{\om_\psi^n} -\dfrac{n\omega_{\f }^{n-1}\wedge \theta }{\omega_\f ^n} - \beta\frac{\om^n}{\om_\f^n} \right) \\
      &=& -\int_0^1\dfrac{d }{d s}\left( \dfrac{n\omega_{s }^{n-1}\wedge \theta }{\omega_{s} ^n} + \beta\frac{\om^n}{\om_s ^n}\right) ds\\
      &= & \int_0^1 \eta_s^{\bar \ell k}(\omega_\f- \omega_\psi )  ds\\
      &=& \left(\int_0^1 \eta_s^{k\bar \ell }ds\right)  \p_{k}\p_{\bar \ell} \phi_\delta -\delta \left(\int_0^1 \eta_s^{ k\bar \ell }ds\right) (R_h)_{k\bar \ell}. 
 \end{eqnarray}
Since $\sup_{X\setminus D}|\dot \f |\leq C$ and $\sup_{X\setminus D}|\dot \psi |\leq C$ we have $\omega_s\geq C^{-1}\theta$, and
$$\beta\frac{\om^n}{\om_u^n}= c_\beta- \dot u-\tr_{\om_u}\theta \leq C, \text{ on } X\setminus D,$$
where $u$ is $\f$ or $\psi$. Therefore we also have
$$\beta\frac{\om^n}{\om_s^n}\leq C, \text{ on 
} X\setminus D.$$
It follows   that   $(\eta_s^{\bar \ell k}  )\leq C\theta^{\bar \ell k}$ on $X\setminus D$.  Therefore
$$\delta \left(\int_0^1 \eta_s^{k\bar \ell }ds\right) (R_h)_{k\bar \ell} \leq C\delta g^{k\bar \ell }(R_h)_{k\bar \ell} \leq \frac{2C\delta}{\varepsilon_0},$$
 since $\theta-\varepsilon_0R_h>0$ for a uniform constant $\varepsilon_0$. We thus obtain
 
 \[\frac{\p}{\p t}\phi_\delta \geq \left(\int_0^1 \eta_s^{k\bar \ell }ds\right)  \p_{k}\p_{\bar \ell} \phi_\delta -\frac{2C\delta}{\varepsilon_0}.\]
 The maximum principle, on any time interval $[0,T]$ now yields
 $$\phi_\delta  \geq -C_1 \delta t\geq -C_1 \delta T ,$$
 for a uniform constant $C_1$. This implies that $\f\geq \psi +\delta |s|^2_h -C_1\delta T$, so $\f\geq \psi$ on $X\setminus D$ by letting $\delta\rightarrow 0$. The same argument shows that $\psi\geq \f$ on $X\setminus D$. Therefore $\f=\psi$ on $X\setminus D$, hence on $X$. 
 \end{proof}
  
 \begin{proof}[Proof of Corollary \ref{cor:j_bounded_below}]
 For any $\f_0\in \mathcal{H}_\omega$, since $\J^\theta_{\omega,\beta}$ is decreasing  along the J-flow by \eqref{J_decreasing}, we have 
  $ \J^\theta_{\omega,\beta}(\f_0)\geq\lim_{t\rightarrow \infty} \J^\theta_{\omega,\beta} (\f_t)$.  Arguing as in  \cite[Lemma 3.2]{FLSW}, we obtain  $$\lim_{t\rightarrow \infty} \J^\theta_{\omega,\beta} (\f_t)=\J^\theta_{\omega,\beta} (\f_\infty).$$
  
By  Theorem \ref{thm:SW-Z}, for any $\e\in (0, \epsilon_0]$, $\f_\e(t)$ converge uniformly to $\f_{\e, \infty}$ which satisfies 
 $$\beta\om^n+ n\om_{\f_{\e, \infty}}^{n-1}\wedge\theta_{\e}= c_\e\om^n_{\f_{\e, \infty}}.$$
 It follows from  Proposition \ref{prop:C0_elliptic}  that $Osc(\f_{\e, \infty}):= \sup_X \f_{\e, \infty}-\inf_X \f_{\e, \infty}\leq C$ for some $C>0$ only depends on $X, \omega, \|\theta\|_{C^0(X, \om)}, \beta$. This implies that  $Osc(\f_{\infty}) $ is uniformly bounded by a constant  depending only on $X, \omega, \|\theta\|_{C^0(X, \om)}, \beta$. 
 
 Since  for any constant  $C$, $\J^\theta_{\omega,\beta} (\f+C)=\J^\theta_{\omega,\beta}(\f), \forall \f\in PSH(X,\om)  $, we infer that $\J^\theta_{\omega,\beta} (\f_\infty)\geq B$ for some uniformly constant $B$ which  only depends on $X, \omega, \|\theta\|_{C^0(X, \om)}, \beta$. Hence we get the uniform lower bound for $\J^\theta_{\omega,\beta} $ on $\mathcal{H}_\om$ as required.
 \end{proof}

\begin{proof}[Proof of Corollary \ref{cor:twisted_cscK}]  Since $\J_{\om,\beta}^{\theta}= \J_{\om,\beta}^{\theta-\eta} +\J_\om^{\eta} $ and 
$|\J_{\om,\beta}^{\theta-\eta}-\J^{-Ric(\om)}_{\om, \beta}|\leq C$ because  $\theta-\eta$ and $-Ric(\om)$ are in the same cohomology class (cf. \cite[Proposition 22]{CS}), we infer that the Mabuchi K-energy twisted by $\eta$ 
\begin{eqnarray}
    \M^\eta_\om= \M_\om+\J^{\eta}_\om&=& (\M_{ent}-\beta J_\om)+\J_{\om, \beta}^{-Ric(\om)} +\J^{\eta}_\om\\
    &\geq &(\M_{ent}-\beta J_\om)+\J_{\om, \beta}^{\theta-\eta} +\J^{\eta}_\om- C\\
    &=&(\M_{ent}-\beta J_\om)+\J_{\om, \beta}^{\theta} - C.
\end{eqnarray}
We consider the degenerate twisted J-flow with $\theta\in -c_1(X)+\eta$ and $\om$ in the assumption.
Then by Corollary \ref{cor:j_bounded_below}, $\J^\theta_{\omega,\beta}$ is uniformly bounded from below on $\mathcal{H}_\om$. Combining with the fact that 
$\M_{ent}\geq \alpha_0 J_\om-C$ and $\alpha_0>\beta$, we infer that the   functional $\M_{\om}^\eta$ is proper. Then Theorem 4.2 in \cite{CC21b} implies the existence of cscK metric with respect to $\eta$.
\end{proof}

We remark that another sufficient  criterion for the properness  of  a twisted Mauchi K-energy can be obtained following the argument for the untwisted version in \cite{LSY} (see also \cite[Lemma  2.1]{JSS}).  We include a sketch of the proof  for the reader convenience.

\begin{proposition}\label{prop:twisted_cscK_crt_2}
Let $X$
 be a compact K\"ahler manifold and $\eta$ 
 be a smooth closed semipositive  $(1,1)$-form. 
Suppose that 
\begin{equation}
    \M_{ent}(\f)\geq \alpha_1(I_\om - J_\om ) (\f)-C, \, \forall \f\in \mathcal{H}_\om,
\end{equation}
for some $\alpha_1>0$, where 
$$I_\om(\f)= \frac{1}{n!}\int_X \f(\omega^n-\omega_\f^n).$$

If  there exists $\epsilon\in [0,\alpha_1)$, a K\"ahler form $\hat\omega \in [\omega]$ and a $(1,1)$-form $\theta\in -c_1(K_X)+[\eta]$ such that $\theta+\epsilon\hat \om>0$ and 
\begin{equation}\label{eq:cond_twisted_2}
    \left( \left[ \frac{n(-c_1(X)+[\eta])\cdot [\omega]^{n-1}}{[\omega]^n}+ \epsilon \right]\hat\omega -(n-1)\theta \right)\wedge  \hat\omega^{n-2}>0.
\end{equation}

Then the  Mabuchi functional twisted by $\eta$ is proper on  $\mathcal{H}_\omega$.  Therefore by \cite[Theorem 4.1]{CC21b} there exists a twisted  cscK metric in $[\om]$. 
\end{proposition}
\begin{proof}
We follow the arguments in \cite{LSY} and  \cite[Lemma  2.1]{JSS}. By direct computation we have  $\J^{\omega}_{\omega}= (I_\om-J_\om) $. 
Therefore we obtain
\begin{equation}
    \M_{ent}(\f)\geq \alpha_1 \J^\om_\om-C\geq  \J^{\epsilon\om}_\om-C\,
\end{equation}
for any $\epsilon\in [0, \alpha_1)$.  On the other hand we have
$$\J^{\epsilon\om}_\om+ \J^{-Ric(\om)}_{\om}+ \J^{\eta}_\om=\J^{\epsilon\om-Ric(\om)+\eta}_{\omega}.$$
Since $\epsilon\om -Ric(\om)+\eta$ and $\epsilon\hat\om +\theta $ are in the same cohomology class, we infer that 
$$|\J^{\epsilon\om-Ric(\om)+\eta}_{\omega} -\J^{\theta+\epsilon\hat\om}_\om|\leq C$$
for some uniform constant $C$. It follows from the main theorem of Song-Weinkove \cite{SW08} that the condition \eqref{eq:cond_twisted_2} implies  that $\J^{\theta+\epsilon \hat \om }_\om$ is uniformly bounded from below. Therefore the Mabuchi functional twisted by $\eta$, $\M^\eta_\om=\M_\om+\J^\eta_\om$, is proper on $\mathcal{H}_\om$. 
\end{proof}

\begin{remark}\label{remark_final}
It follows from \cite[Lemma 6.19, Remark 6.20]{Tian00} that 
$$ \frac{1}{n}J_\om\leq I_\om-J_\om\leq nJ_\omega.$$ Therefore this criterion implies the one in Corollary \ref{cor:twisted_cscK} for $\beta\in [0, \alpha_0/n)$. 
\end{remark}

\end{document}